\documentclass[a4paper, 12pt, reqno]{article}
\usepackage{amsmath}
\usepackage{amsthm}
\usepackage{amssymb}
\usepackage{amsfonts}
\usepackage{mathrsfs}
\usepackage{mathtools}
\usepackage{wrapfig}
\usepackage{graphicx}
\usepackage{cite}
\usepackage{tikz}
\usepackage{hyperref}
\usepackage{comment}
\usepackage{enumitem}
\usepackage{geometry}
\numberwithin{equation}{section}

\theoremstyle{mystyle}
\newtheorem{thm}{\textbf{Theorem}}[section]
\newtheorem*{thm*}{\textbf{Theorem}}
\newtheorem{prop}[thm]{\textbf{Proposition}}
\newtheorem*{prop*}{\textbf{Proposition}}
\newtheorem{lem}[thm]{\textbf{Lemma}}
\newtheorem*{lem*}{\textbf{Lemma}}

\theoremstyle{definition}

\newtheorem*{dfn*}{\textbf{Definition}}
\newtheorem{rem}[thm]{\textbf{Remark}}
\newtheorem*{rem*}{\textbf{Remark}}

\mathtoolsset{showonlyrefs}

\newcommand{\dualC}[2]{
\langle #1 , #2 \rangle
}

\newcommand{\norm}[2]{
\| #1 \|_{#2}
}
\newcommand{\set}[2]{
\{ #1 : #2 \}
}
\newcommand{\SET}[2]{
\left\{ #1 : #2 \right\}
}
\newcommand{\bR}{\mathbb{R}}
\newcommand{\bN}{\mathbb{N}}
\newcommand{\FT}{\mathscr{F}}

\title{On solitary wave solutions to dispersive equations with double power nonlinearities}
\author{Kaito Kokubu\footnote{Email: \texttt{1123701@ed.tus.ac.jp}}}
\date{}

\AtBeginDocument{
   \def\MR#1{}
}

\begin{document}

\maketitle 
\begin{center}
    Department of Mathematics, Graduate School of Science, Tokyo University of Science, \\
    1-3 Kagurazaka, Shinjuku-ku, Tokyo 162-8601, Japan
\end{center}

\begin{abstract}
    We study semilinear elliptic equations with the fractional Laplacian in $\bR$.
    The equations with single power nonlinearities have been observed by Weinstein(1987), Frank--Lenzmann(2013) and so on.
    We focus on the equations with double power nonlinearities and consider the existence of ground states.
\end{abstract}

\section{Introduction}

We consider the following stationary problem:
\begin{align}
    D_{x}^{\sigma} \phi + c \phi - f(\phi) = 0, \quad x\in\bR, \label{eq:SP0}
\end{align}
where $\phi$ is a real-valued unknown function, $c$ is a positive constant and $f$ is some given nonlinearity.
Furthermore, for $0 < \sigma < 2$, we define the operator $D_{x}^{\sigma}$ as $D_{x}^{\sigma}:=\FT^{-1}|\xi|^{\sigma}\FT$, where $\FT$ is the Fourier transform.
This operator is often written as $(-\Delta)^{\sigma/2}$.

It is well known that the solutions of \eqref{eq:SP0} provide solitary wave solutions of some dispersive equations.
For example, suppose that $\phi$ is a solution to \eqref{eq:SP0}, then $u(t,x)=\phi(x-ct)$ is a travelling wave solution to Benjamin--Ono type equations
\begin{align}
  \partial_{t}u + \partial_{x}f(u) + \partial_{x}D_{x}^{\sigma}u = 0, \quad x\in \bR, \ t\in \bR,
\end{align}
with the wave speed $c>0$, and $u(t,x)=e^{i \omega t}\phi(x)$ is a standing wave solution to fractional Schr\"{o}dinger equations
\begin{align}
  i\partial_{t}u - D_{x}^{\sigma}u + f(u) = 0, \quad x\in \bR, \ t\in \bR,
\end{align}
with the frequency of oscillation $\omega>0$.
(Here we replace $c$ with $\omega$.)

One of the important solutions is \textit{a ground state}.
A ground state is a nontrivial solution to \eqref{eq:SP0} which minimizes the value of  \textit{the action functional} corresponding to \eqref{eq:SP0}:
\begin{align}
  S(u) = \frac{1}{2}\norm{D_{x}^{\sigma/2}u}{L^{2}}^{2} + \frac{c}{2}\norm{u}{L^{2}}^{2} - \int_{\bR} F(u) \, dx,
\end{align}
where $F(s):=\int_{0}^{s} \, f(t) dt$ for $s\in\bR$.

The stationary problems \eqref{eq:SP0} with single power nonlinearities have been studied well so far.
First, in case $\sigma=1$, Benjamin\cite{Benjamin} obtained the explicit solution to \eqref{eq:SP0} with $f(s) = s^{2}$: $\phi(x) = \frac{2c}{1+c^{2}x^{2}}$.
Furthermore, Weinstein\cite{Weinstein} found a positive and even ground state of \eqref{eq:SP0} with $f(s)=s^{p} \ (p\in\mathbb{N}, \ p\geq2.)$
With the same method as \cite{Weinstein}, we can obtain the existence of a positive and even ground state of more general problems
\begin{align}
  D_{x}^{\sigma} \phi + c \phi - |\phi|^{p-1}\phi = 0,
\end{align}
where $0 <\sigma < 2, \ 1 < p < 2^{\ast}_{\sigma}-1$, and $2^{\ast}_{\sigma}:=2/(1-\sigma)_{+}$ (see Remark (1) of Proposition 1.1 in \cite{Frank-Lenzmann}.)
The uniqueness of ground states of these equations is shown by Frank--Lenzmann\cite{Frank-Lenzmann}.

Recently, the stability and instability of solitary wave solutions of dispersive equations with double power nonlinearities have been studied: for example, Ohta\cite{Ohta_1995_S}, Fukaya-Hayashi\cite{Fukaya-Hayashi-2011}, and Lewin--Nodari \cite{Lewin-Nodari} for the Schr\"{o}dinger equation.
Then our interest is the stability and instability of travelling wave solutions to the Benjamin--Ono equation
\begin{align}
  \partial_{t} u + \partial_{x}(-\phi^{p} + \phi^{q}) + \partial_{x}D_{x}u = 0, \quad x\in\bR, \ t\in\bR, \tag{GBO} \label{eq:GBO}
\end{align}
where $p, \ q \in\mathbb{N}, \ 2 \leq p < q$.
To observe this, we need to study the existence of solutions, especially ground states, of the following stationary problem derived from \eqref{eq:GBO}:
\begin{align}
  D_{x}\phi + c\phi + \phi^{p} - \phi^{q} = 0, \quad x\in\bR. \tag{SP} \label{eq:SP00}
\end{align}
The results on the existence depend on the parities of $p$ and $q$.
Therefore, the aim of this paper is to classify the existence and properties of solutions to \eqref{eq:SP00} with respect to the parities.
To consider the classification, we first observe \eqref{eq:SP0} with the following double power nonlinearities:
\begin{align}
  & f_{1}(s) := -|s|^{p-1}s + |s|^{q-1}s, \\
  & f_{2}(s) := |s|^{p-1}s + |s|^{q-1}s, \\
  & f_{3}(s) := |s|^{p-1}s - |s|^{q-1}s,
\end{align}
for $s\in\bR$, where $1 < p < q < 2^{\ast}_{\sigma} -1$, and apply these results to the classification of solutions of \eqref{eq:SP}.

\section{Main Results} \label{section:main_results}
First, for convenience, we name equations \eqref{eq:SPj} as
\begin{align}
  D_{x}^{\sigma} \phi + c \phi - f_{j}(\phi) = 0, \quad x\in\bR. \tag{SP$j$} \label{eq:SPj}
\end{align}
Moreover, we put
\begin{align}
  2^{\ast}_{\sigma} := \frac{2}{(1-\sigma)_{+}} = \left\{
    \begin{aligned}
     & \frac{2}{1-\sigma}, \ &\text{for} \ 0 < \sigma < 1, \\
     & \infty, \ &\text{for} \ 1 \leq \sigma < 2.
    \end{aligned}
  \right.
\end{align}
Before stating the main results, we define a ground state.
First, we write the action funcional $S_{j}\colon H^{\sigma/2}(\bR) \rightarrow \bR$ corresponding to \eqref{eq:SPj} as
\begin{align}
  S_{j}(u) := \frac{1}{2}\norm{D_{x}^{\sigma/2}u}{L^{2}}^{2} + \frac{c}{2}\norm{u}{L^{2}}^{2} - \int_{\bR} F_{j}(u) \, dx,
\end{align}
where
\begin{align}
  F_{j}(s):=\int_{0}^{s} f_{j}(t) \, dt
\end{align}
for $s\in\bR$.
We can see that $S_{j}\in C^{1}(H^{\sigma/2}(\bR),\bR)$ and
\begin{align}
  S_{j}'(u) = D_{x}^{\sigma}u + cu - f_{j}(u),
\end{align}
which implies that $u\in H^{\sigma/2}(\bR)$ is a solution to \eqref{eq:SP0} if and only if $S_{j}'(u)=0$.
Here we put the set of nontrivial solutions of \eqref{eq:SPj} as
\begin{align}
  \mathcal{N}_{j}:= \set{v\in H^{\sigma/2}(\bR) \setminus \{ 0 \}}{S_{j}'(v)=0}.
\end{align}

\begin{dfn*}
  By {\it a ground state}, we mean $\phi\in\mathcal{N}_{j}$ which minimizes the value of $S_{j}$ in $\mathcal{N}_{j}$.
  We put the set of ground states of \eqref{eq:SPj} as
  \begin{align}
    \mathcal{G}_{j} := \set{ \phi\in\mathcal{N}_{j} }{ S_{j}(\phi) \leq S_{j}(v) \ \text{for all} \ v\in\mathcal{N}_{j} }.
  \end{align}
\end{dfn*}

Here we state the results for \eqref{eq:SPj}.

\begin{thm}
  Set $j=1, \ 2$ and let $0 < \sigma < 2$.
  For any $c>0$, there exists a ground state of \eqref{eq:SPj} which belongs to $H^{\sigma + 1}(\bR)$.
  Moreover, it can be taken as positive, even, and decreasing in $|x|$.
  In addition, there do not exist any ground states of \eqref{eq:SPj} with sign changes. \label{Thm:SP1_2}
\end{thm}

\begin{thm}
  Let $0 < \sigma \leq 1$ and set 
  \begin{align}
    \alpha := \frac{q-1}{q-p}, \quad \beta := \frac{p-1}{q-p}, \quad c_{0} := \frac{2(q-p)(p-1)^{\beta}(q+1)^{\beta}}{(p+1)^{\alpha}(q-1)^{\alpha}}.
  \end{align}
  If $c\in(0,c_{0})$, then there exists a ground state of \eqref{eq:SP3} which belongs to $H^{\sigma+1}(\bR)$ and is positive, even, and decreasing in $|x|$. \label{Thm:SP3}
\end{thm}

\begin{rem}
  In Theorem \ref{Thm:SP3}, we do not deny the existence of a ground state of \eqref{eq:SP3} which has sign changes.
\end{rem}

\begin{rem}
  We can also obtain a negative ground state of \eqref{eq:SPj} if it has a positive ground state.
  Indeed, let $\phi\in\mathcal{G}_{j}$ be positive.
  Then we obtain $-\phi\in\mathcal{N}_{j}$ and $S_{j}(-\phi)=S_{j}(\phi)$, which means $-\phi\in\mathcal{G}_{j}$.
\end{rem}

Applying Theorems \ref{Thm:SP1_2} and \ref{Thm:SP3}, we obtain the classification results of \eqref{eq:SP}:
\begin{thm}
  Let $c>0, \ p,\ q \in \mathbb{N}, \ 2 \leq p < q$.
  Then the followings hold for \eqref{eq:SP}:
  \begin{enumerate}
    \item Assume that $p$ is odd and $q$ is odd.
      Then there exists a ground state which is positive, even, and decreasing in $|x|$.
      There also exists a ground state which is negative, even, and increasing in $|x|$.  \label{case:odd_odd}
    \item Assume that $p$ is odd and $q$ is even. 
      Then there exists a ground state which is positive, even, and decreasing in $|x|$, while there do not exist any negative solutions. \label{case:odd_even}
    \item Assume that $p$ is even and $q$ is odd.
      Then there exists a solution which is positive, even, and decreasing in $|x|$.
      Moreover, there exists a ground state which is negative, even, and increasing in $|x|$.
      In addition, none of the positive solutions is a ground state. \label{case:even_odd}
    \item Assume that $p$ is even and $q$ is even.
      Then there exists a solution which is positive, even, and decreasing in $|x|$.
      Moreover, there exists a solution which is negative, even, and increasing in $|x|$ if $c \in (0,c_{0})$, where $c_{0}$ is defined in Theorem \ref{Thm:SP3}. \label{case:even_even}
  \end{enumerate} \label{thm:classification}
\end{thm}

\begin{rem}
  In Case (\ref{case:even_even}) of Theorem \ref{thm:classification}, ground states are not identified yet.
\end{rem}

In the following, we will split the proof of Theorems \ref{Thm:SP1_2} and \ref{Thm:SP3} into two parts.
We prove the existence of ground states of \eqref{eq:SPj} in \S\ref{Section:existence_of_GS}, and the regularity and positivity of solutions of \eqref{eq:SPj} in \S \ref{section:properties}.
Finally, the proof of Theorem \ref{thm:classification} will be given in \S\ref{section:application}.

\section{Existence of ground states} \label{Section:existence_of_GS}

In this section, we consider the existence of ground states of \eqref{eq:SPj}.
Before considering the existence, we recall the symmetric decreasing rearrangement of a nonnegative function, which plays an important role to obtain a solution which is positive, even, and decreasing in $|x|$.

Let $u \in H^{\sigma/2}(\bR)$.
Then we obtain $|u| \in H^{\sigma/2}(\bR)$ and the following inequality holds:
\begin{align}
  \norm{D_{x}^{\sigma/2}|u|}{L^{2}} \leq \norm{D_{x}^{\sigma/2}u}{L^{2}}. \label{eq:A_101}
\end{align}
Therefore, we can define the symmetric decreasing rearrangement $|u|^{\ast}$ of $|u|$, and we obtain
\begin{align}
  \norm{D_{x}^{\sigma/2}|u|^{\ast}}{L^{2}} \leq \norm{D_{x}^{\sigma/2}|u|}{L^{2}}. \label{eq:A_102}
\end{align}
We can prove these properties similarly to \cite[Lemma 8.15]{Angulo}.
In the following, we simply denote $u^{\ast}:=|u|^{\ast}$ for $u\in H^{\sigma/2}(\bR)$.

\subsection{Case $j=1,2$} \label{subsection:Nehari}
In this subsection, we prove the following proposition.
\begin{prop}
  For $j=1,2$, there exists $\phi\in\mathcal{G}_{j}$ which is nonpositive, even, and decreasing in $|x|$. \label{prop:existense_of_gs_12}
\end{prop}

To simplify the discussion of Proposition \ref{prop:existense_of_gs_12}, we focus on the equation
\begin{align}
  D_{x}^{\sigma}\phi + c\phi + |\phi|^{p-1}\phi - |\phi|^{q-1}\phi = 0, \quad x\in\bR. \tag{SP1} \label{eq:SP1}
\end{align}
We can prove the existence for (SP2) in the same way.
In this case, the action functional $S_{1}$ is written as
\begin{align}
  S_{1}(u) = \frac{1}{2}\norm{D_{x}^{\sigma/2}u}{L^{2}}^{2} + \frac{c}{2}\norm{u}{L^{2}}^{2} + \frac{1}{p+1}\norm{u}{L^{p+1}}^{p+1} - \frac{1}{q+1}\norm{u}{L^{q+1}}^{q+1}. \label{eq:def_action_SP1}
\end{align}
Then using this functional, we define the Nehari functional $K_{1}$ as
\begin{align}
  K_{1}(u) &:= \dualC{S_{1}'(u)}{u} \\
  &= \norm{D_{x}^{\sigma/2}u}{L^{2}}^{2} + c\norm{u}{L^{2}}^{2} + \norm{u}{L^{p+1}}^{p+1} - \norm{u}{L^{q+1}}^{q+1}, \quad u\in H^{\sigma/2}(\bR), \label{eq:def_nehari_SP1}
\end{align}
and set
\begin{align}
  \mathcal{K}_{1} := \set{ v\in H^{\sigma/2}(\bR) \setminus \{ 0 \} }{ K_{1}(v)=0 }.
\end{align}
By the definition of $K_{1}$, the inclusion $\mathcal{N}_{1}\subset\mathcal{K}_{1}$ obviously holds.

First, show that $\mathcal{G}_{1}$ is not empty.
To prove this, we introduce the following value and set:
\begin{align}
  d_{1}:= \inf_{v\in\mathcal{K}_{1}} S_{1}(v), \quad
  \mathcal{M}_{1} := \set{ v\in\mathcal{K}_{1} }{S_{1}(v)=d_{1}}.
\end{align}

\begin{lem}
    If $\mathcal{M}_{1}$ is not empty, then $\mathcal{M}_{1}=\mathcal{G}_{1}$ holds. \label{lem:relation_G_and_M}
\end{lem}
\begin{proof}
  First, we show $\mathcal{M}_{1}\subset\mathcal{G}_{1}$.
  Let $\phi \in \mathcal{M}_{1}$.
  Then we can see that $K_{1}(\phi) = 0$.
  Here we let $\lambda > 0$ and consider the function
  \begin{align}
    \lambda \mapsto K_{1}(\lambda \phi) = \lambda^{2} \left( \norm{D_{x}^{\sigma/2}\phi}{L^{2}}^{2} + c\norm{\phi}{L^{2}}^{2} \right) + \lambda^{p+1} \norm{\phi}{L^{p+1}}^{p+1} - \lambda^{q+1}\norm{\phi}{L^{q+1}}^{q+1}. \label{function:K1}
  \end{align}
  Considering the graph of \eqref{function:K1}, we obtain
  \begin{align}
    \left. \partial_{\lambda} K_{1}(\lambda\phi) \right|_{\lambda = 1} = \dualC{K_{1}'(\phi)}{\phi} < 0, \label{eq:A_2}
  \end{align}
  which implies $K_{1}'(\phi) \neq 0$.
  Then there exists a Lagrange multiplier $\mu \in \bR$ such that $S_{1}'(\phi) = \mu K_{1}'(\phi)$.
  From \eqref{eq:A_2} and $0 = K_{1}(\phi) = \dualC{S_{1}'(\phi)}{\phi} = \mu \dualC{K_{1}'(\phi)}{\phi}$, we see that $\mu = 0$, which implies $\phi\in\mathcal{N}_{1}$.
  Furthermore, we have $K_{1}(v) = 0$ for any $v\in\mathcal{N}_{1}$ and, by the definition of $\mathcal{M}_{1}$, we obtain $S_{1}(\phi) = d_{1} \leq S_{1}(v)$,
  which means $\phi\in\mathcal{G}_{1}$.

  Next, we show $\mathcal{G}_{1}\subset\mathcal{M}_{1}$.
  By the assumption of this lemma, we can take some $v\in\mathcal{M}_{1}$, and we have shown $\mathcal{M}_{1}\subset\mathcal{G}_{1}$ above.
  Then we have $S_{1}(\phi) \leq S_{1}(v) = d_{1}$ for any $\phi\in\mathcal{G}_{1}$.
  On the other hand, from $\phi\in\mathcal{K}_{1}$ and the definition of $d_{1}$, we can see that $d_{1} \leq S_{1}(\phi)$.
  This implies $\mathcal{G}_{1}\subset\mathcal{M}_{1}$.
\end{proof}

\begin{rem}
  It is difficult to verify the same claim as Lemma \ref{lem:relation_G_and_M} for the equation
  \begin{align}
    D_{x}^{\sigma}\phi + c \phi - |\phi|^{p-1}\phi + |\phi|^{q-1}\phi = 0 \tag{SP3} \label{eq:SP3}
  \end{align}
  because we hardly obtain the condition independent of $\phi\in\mathcal{M}_{3}$ that gives the same shape graph of $K_{3}(\lambda \phi)$ as $K_{1}(\lambda\phi)$
  (The definitions of $K_{3}$ and $\mathcal{M}_{3}$ are similar to those of $K_{1}$ and $\mathcal{M}_{1}$, respectively.)
  Therefore, we do not apply the same method to the existence of a ground state of \eqref{eq:SP3}. \label{rem:difficulty}
\end{rem}

Thanks to Lemma \ref{lem:relation_G_and_M}, we shall show that $\mathcal{M}_{1}$ is not empty to obtain the existence of a ground state of \eqref{eq:SP1}.
However, we claim a stronger statement.

\begin{lem}
  Let $(u_{n})_{n}$ be a minimizing sequence of $d_{1}$, that is, $(u_{n})_{n}\subset H^{\sigma/2}(\bR)$ satisfies $S_{1}(u_{n})\rightarrow d_{1}$ and $K_{1}(u_{n})\rightarrow 0$.
  Then there exists a sequence $(z_{n})_{n}\subset\bR$ such that, by taking a subsequence, we find $v\in\mathcal{M}_{1}$ which satisfies $u_{n}(\cdot + z_{n}) \rightarrow v$ in $H^{\sigma/2}(\bR)$. \label{prop:existence_of_M1}
\end{lem}

To prove Proposition \ref{prop:existence_of_M1}, we need the following two lemmas.

\begin{lem}[Lieb\cite{Lieb}]
  Let $(u_{n})_{n}$ be a bounded sequence in $H^{\sigma/2}(\bR)$.
  Moreover, we assume that there exists $r\in(2,\infty)$ such that $\inf_{n\in\bN} \norm{u_{n}}{L^{r}} > 0$.
  Then there exists $(z_{n})_{n}\subset\bR$ such that, by taking a subsequence, we find $v\in H^{\sigma/2}(\bR)\setminus \{ 0 \}$ which satisfies $u_{n}(\cdot + z_{n}) \rightharpoonup v$ weakly in $H^{\sigma/2}(\bR)$. \label{lem:Lieb}
\end{lem}

Originally, Lieb\cite{Lieb} showed the same statement for a sequence in $H^{1}(\bR)$. 
For the proof, see Appendix \ref{appendix:proof_of_Lieb}.

\begin{lem}[Brezis--Lieb\cite{Brezis-Lieb}]
  Let $r\in(1,\infty)$ and $(u_{n})_{n}$ be a bounded sequence in $L^{r}(\bR)$.
  Moreover, we assume that there exists a measurable function $u\colon \bR \rightarrow \bR$ such that $u_{n} \rightarrow u$ a.e.\ in $\bR$.
  Then $u\in L^{r}(\bR)$ and we have
  \begin{align}
    \lim_{n\rightarrow\infty} \left( \int_{\bR}|u_{n}|^{r} \, dx - \int_{\bR}|u_{n}-u|^{r} \, dx \right) = \int_{\bR} |u|^{r} \, dx.
  \end{align}
  \label{lem:Brezis-Lieb}
\end{lem}

Before proceeding with the proof of Proposition \ref{prop:existence_of_M1}, we define the functional $I_{1}\colon H^{\sigma/2}(\bR)\rightarrow \bR$ as
\begin{align}
  I_{1}(u) := \left( \frac{1}{2} - \frac{1}{q+1} \right) \norm{u}{H^{\sigma/2}_{c}}^{2} + \left( \frac{1}{p+1} - \frac{1}{q+1} \right) \norm{u}{L^{p+1}}^{p+1}, \label{eq:A_3}
\end{align}
where $\norm{u}{H^{\sigma/2}_{c}}^{2}:=\norm{D_{x}^{\sigma/2}u}{L^{2}}^{2} + c\norm{u}{L^{2}}^{2}$ for $u\in H^{\sigma/2}(\bR)$ and $c$ is the positive constant included in \eqref{eq:SPj}.
Using \eqref{eq:A_3}, we can write
\begin{align}
  S_{1}(u) = \frac{1}{q+1} K_{1}(u) + I_{1}(u), \label{eq:A_4}
\end{align}
which implies that
\begin{align}
  d_{1} = \inf_{v\in\mathcal{K}_{1}}S_{1}(v) = \inf_{v\in\mathcal{K}_{1}}I_{1}(v). \label{eq:A_5}
\end{align}
Here we obtain the following two lemmas.

\begin{lem}
  Let $v\in H^{\sigma/2}(\bR)$ satisfy $K_{1}(v)<0$. Then $I_{1}(v) > d_{1}$ holds. \label{lem:negativeness_of_I1}
\end{lem}
\begin{proof}
  Considering the shape of the graph of $ (0,\infty) \ni \lambda \mapsto K_{1}(\lambda v)$, there exists $\lambda_{0}\in(0,1)$ such that $K_{1}(\lambda_{0} v)=0$, which means $\lambda_{0} v \in \mathcal{K}_{1}$.
  Finally, by the definition of $d_{1}$, we obtain $d_{1} \leq I_{1}(\lambda_{0}v) < I_{1}(v)$.
\end{proof}

\begin{lem}
  $d_{1}>0$. \label{lem:positiveness_of_d1}
\end{lem}
\begin{proof}
  By \eqref{eq:A_5}, we shall show that there exists $C_{0}>0$ such that $I_{1}(v) \geq C$ holds for any $v\in\mathcal{K}_{1}$.

  Let $v\in\mathcal{K}_{1}$.
  By the Sobolev embedding, we have
  \begin{align}
    0 = K_{1}(v) &= \norm{v}{H^{\sigma/2}_{c}}^{2} + \norm{v}{L^{p+1}}^{p+1} - \norm{v}{L^{q+1}}^{q+1} \\
    &\geq \norm{v}{H^{\sigma/2}_{c}}^{2}\left( 1 - C_{1}\norm{v}{H^{\sigma/2}_{c}}^{p-1} - C_{2} \norm{v}{H^{\sigma/2}_{c}}^{q-1} \right),
  \end{align}
  which implies $1 \leq C_{1}\norm{v}{H^{\sigma/2}_{c}}^{p-1}$ or $ 1 \leq C_{2}\norm{v}{H^{\sigma/2}_{c}}^{q-1} $.
  Here we put
  \begin{align}
    C_{0} := \left( \frac{1}{2} - \frac{1}{q+1} \right) \min \left\{ C_{1}^{-\frac{2}{p-1}}, \ C_{2}^{-\frac{2}{q-1}} \right\},
  \end{align}
  and then we obtain
  \begin{align}
    I_{1}(v) \geq \left( \frac{1}{2} - \frac{1}{q+1} \right) \norm{v}{H^{\sigma/2}}^{2} \geq C_{0},
  \end{align}
  which is the desired inequality.
\end{proof}

Now, we prove Proposition \ref{prop:existence_of_M1}.
\begin{proof}[Proof of Lemma \ref{prop:existence_of_M1}]
  First, we remark that, by using the Nehari functional $K_{1}$, we can write $S_{1}$ as
  \begin{align}
    S_{1}(u) = \frac{1}{2} K_{1}(u) - \left( \frac{1}{2} - \frac{1}{p+1} \right) \norm{u}{L^{p+1}}^{p+1} + \left( \frac{1}{2} - \frac{1}{q+1} \right) \norm{u}{L^{q+1}}^{q+1}. \label{eq:A_6}
  \end{align}
  Let $(u_{n})_{n}$ be a sequence in $H^{\sigma/2}(\bR)$ which satisfies $S_{1}(u_{n}) \rightarrow d_{1}$ and $K_{1}(u_{n}) \rightarrow 0$.
  From \eqref{eq:A_4}, \eqref{eq:A_6} and the assumption, we obtain
  \begin{align}
    & I_{1}(u_{n}) \rightarrow d_{1}, \label{eq:A_7} \\
    & -\left( \frac{1}{2} - \frac{1}{p+1} \right)\norm{u_{n}}{L^{p+1}}^{p+1} + \left( \frac{1}{2} - \frac{1}{q+1} \right)\norm{u_{n}}{L^{q+1}}^{q+1} \rightarrow d_{1}. \label{eq:A_8}
  \end{align}
  Then \eqref{eq:A_7} implies that $(u_{n})_{n}$ is bounded in $H^{\sigma/2}(\bR)$.
  Furthermore, by \eqref{eq:A_8}, we can see that $\inf_{n\in\bN}\norm{u_{n}}{L^{p+1}}>0$.
  Indeed, if $\inf_{n\in\bN}\norm{u_{n}}{L^{p+1}}=0$, it contradicts to the positivity of $d_{1}$.
  Therefore, we can apply Lemma \ref{lem:Lieb} to $(u_{n})_{n}$ and then there exists $(z_{n})_{n}\subset \bR$ and $v\in H^{\sigma/2}(\bR)\setminus \{0\}$ which satisfies $u_{n}(\cdot + z_{n}) \rightharpoonup v$ weakly in $H^{\sigma/2}(\bR)$ by taking a subsequence.
  Here we denote $v_{n}:=u_{n}(\cdot + z_{n})$.
  By the Rellich compact embedding $H^{\sigma/2}(\bR) \hookrightarrow L^{r}_{\textrm{loc}}(\bR)$ for $r\in[2,2^{\ast}_{\sigma})$, we may assume that $v_{n} \rightarrow v$ a.e.\ in $\bR$.
  Then by using Lemma \ref{lem:Brezis-Lieb}, we have the following convergences:
  \begin{align}
    & I_{1}(v_{n}) - I_{1}(v_{n}-v) \rightarrow I_{1}(v), \label{eq:A_9} \\
    & K_{1}(v_{n}) - K_{1}(v_{n}-v) \rightarrow K_{1}(v). \label{eq:A_10}
  \end{align}
  From \eqref{eq:A_9}, we can see that
  \begin{align}
    \lim_{n\rightarrow\infty} I_{1}(v_{n}-v) &= \lim_{n\rightarrow\infty} I_{1}(v_{n}) - I_{1}(v) \\
    &< \lim_{n\rightarrow\infty} I_{c}(v_{n}) = \lim_{n\rightarrow\infty} I_{c}(u_{n}) = d_{1}, \label{eq:A_11}
  \end{align}
  which implies that $I_{1}(v_{n}-v) \leq d_{1}$ holds for large $n\in\bN$ and then we have $K_{1}(v_{n}-v) \geq 0$ by Lemma \ref{lem:negativeness_of_I1}.
  Moreover, from \eqref{eq:A_10}, we obtain
  \begin{align}
    K_{1}(v) = \lim_{n\rightarrow\infty}K_{1}(v_{n}) - \lim_{n\rightarrow\infty}K_{1}(v_{n}-v) \leq 0. \label{eq:A_12}
  \end{align}
  Therefore, applying Lemma \ref{lem:negativeness_of_I1} again, we have
  \begin{align}
    d_{1} \leq I_{1}(v) \leq \liminf_{n\rightarrow\infty}I_{1}(v_{n}) = d_{1},
  \end{align}
  that is,
  \begin{align}
    I_{1}(v) = d_{1}. \label{eq:A_13}
  \end{align}
  Using Lemma \ref{lem:negativeness_of_I1} again, we see that $K_{1}(v) \geq 0$.
  Combining this inequality and \eqref{eq:A_12} gives
  \begin{align}
    K_{1}(v)=0. \label{eq:A_15}
  \end{align}
  Therefore, $v\in\mathcal{M}_{1}$ follows from \eqref{eq:A_5}, \eqref{eq:A_13} and \eqref{eq:A_15}.

  Finally, we can obtain the strong convergence by \eqref{eq:A_9}.
  This completes the proof.
\end{proof}

\begin{lem}
  Let $\phi\in\mathcal{G}_{1}$.
  Then we obtain $\phi^{\ast}\in\mathcal{G}_{1}$. \label{prop:even_ground_state_of_SP1_2}
\end{lem}
\begin{proof}
  Thanks to Lemma \ref{lem:relation_G_and_M}, it suffices to show $\phi^{\ast}\in\mathcal{M}_{1}$.
  From $\phi\in\mathcal{G}_{1} = \mathcal{M}_{1}$, we have $S_{1}(\phi) = d_{1}, \ K_{1}(\phi) = 0$.
  Using the inequalities \eqref{eq:A_101} and \eqref{eq:A_102} yields that
  \begin{align}
    & S_{1}(\phi^{\ast}) \leq S_{1}(\phi) = d_{1}, \label{eq:A_21} \\
    & K_{1}(\phi^{\ast}) \leq 0, \label{eq:A_22} \\
    & I_{1}(\phi^{\ast}) \leq I_{1}(\phi) = d_{1}. \label{eq:A_23}
  \end{align}
  Moreover, we can see that $K_{1}(\phi^{\ast}) \geq 0$ from \eqref{eq:A_23} and Lemma \ref{lem:negativeness_of_I1}.
  Then combining this and \eqref{eq:A_22}, we obtain $K_{1}(\phi^{\ast}) = 0$.
  Therefore, $\phi^{\ast}\in\mathcal{M}_{1}$ follows from this identity and \eqref{eq:A_21}.
  This concludes the proof.
\end{proof}

Finally, Proposition \ref{prop:existense_of_gs_12} follows from Lemmas \ref{lem:relation_G_and_M}, \ref{prop:existence_of_M1}, and \ref{prop:even_ground_state_of_SP1_2}.

\begin{rem}
  \begin{enumerate}
    \item This method is also valid for the case that $\sigma = 2$.
    \item We can also characterize the ground state of \eqref{eq:SP0} with $f(s)=|s|^{p-1}s \ (1 < p < 2^{\ast}_{\sigma}-1)$ with the same method as Proposition \ref{prop:existense_of_gs_12}.
    \item Here we define $f_{4}:=-|s|^{p-1}s -|s|^{q-1}s$ and consider (SP4).
    The Nehari functional $K_{4}$ corresponding to (SP4) is given as \label{rem:SP4}
    \begin{align}
      K_{4}(u) = \dualC{S_{4}'(u)}{u} = \norm{D_{x}^{\sigma/2}u}{L^{2}}^{2} + c\norm{u}{L^{2}}^{2} + \norm{u}{L^{p+1}}^{p+1} + \norm{u}{L^{q+1}}^{q+1} \geq 0.
    \end{align}
    Here we let $v\in H^{\sigma/2}(\bR)$ be a soluion to (SP4).
    Then we obtain $v=0$ by the nonnegativity of the value of $K_{4}$.
  \end{enumerate} \label{rem:end_of_1_2}
\end{rem}

\subsection{Case $j=3$} \label{subsection:Pohizaev}
In this subsection, we prove the following proposition.
\begin{prop}
  Let $0 < \sigma \leq 1$.
  If $c\in(0,c_{0})$, then there exists $\phi\in\mathcal{G}_{3}$ which is nonpositive, even, and decreasing in $|x|$. \label{prop:existence_of_ground_state_of_SP3}
\end{prop}

As mentioned in Remark \ref{rem:difficulty}, it is difficult to find any solutions to \eqref{eq:SP3} with the same method as in Proposition \ref{prop:existense_of_gs_12}.
So, we prove the existence of ground states of \eqref{eq:SP3} with the Pohozaev identity, which is the method introduced by Berestycki--Lions \cite{Berestycki-Lions-1} or Berestycki--Gallou\"{e}t--Kavian \cite{BGK}.

\begin{lem}[The Pohozaev identity]
  The following identity holds for any solution $u\in H^{\sigma/2}(\bR)$ to \eqref{eq:SP3}:
  \begin{align}
    \frac{1-\sigma}{2}\norm{D_{x}^{\sigma/2}u}{L^{2}}^{2} + \frac{c}{2}\norm{u}{L^{2}}^{2} - \frac{1}{p+1}\norm{u}{L^{p+1}}^{p+1} + \frac{1}{q+1}\norm{u}{L^{q+1}}^{q+1} = 0.
  \end{align} \label{lem:Pohozaev}
\end{lem}
\begin{proof}
  See \cite[Proposition 4.1]{Chang-Wang} and the references therein.
\end{proof}

By Lemma \ref{lem:Pohozaev}, we obtain
\begin{align}
  S_{3}(u) = \frac{1}{2}\norm{D_{x}^{\sigma/2}u}{L^{2}}^{2} + \frac{c}{2}\norm{u}{L^{2}}^{2} - \frac{1}{p+1}\norm{u}{L^{p+1}}^{p+1} + \frac{1}{q+1}\norm{u}{L^{q+1}}^{q+1} = \frac{\sigma}{2} \norm{D_{x}^{\sigma/2}u}{L^{2}}^{2}
\end{align}
for $u\in\mathcal{N}_{3}$.
Here we define some functionals to use for the proof.
Let $c>0$.
Then we define
  \begin{align}
    P_{c}(u) := -\frac{c}{2}\norm{u}{L^{2}}^{2} + \frac{1}{p+1}\norm{u}{L^{p+1}}^{p+1} - \frac{1}{q+1}\norm{u}{L^{q+1}}^{q+1}
  \end{align}
for $u\in H^{\sigma/2}(\bR)$, and
\begin{align}
  \mathcal{P}_{c} &:= \left\{
    \begin{aligned}
      & \set{ v\in H^{\sigma/2}(\bR)\setminus \{0\} }{ P_{c}(v)=0 }, & & \text{for} \ \sigma = 1, \label{eq:B_3} \\
      & \set{ v\in H^{\sigma/2}(\bR)\setminus \{0\} }{ P_{c}(v)=1 }, & & \text{for} \ 0 < \sigma < 1.
    \end{aligned}
  \right.
\end{align}

To find a solution to \eqref{eq:SP3}, we consider the following minimizing problem:
\begin{align}
  j_{3} := \inf_{v\in\mathcal{P}_{c}} J_{3}(v), \label{eq:minimizing_problem}
\end{align}
where
\begin{align}
  J_{3}(u):= \frac{\sigma}{2}\norm{D_{x}^{\sigma/2}u}{L^{2}}^{2}
\end{align}
for $u\in H^{\sigma/2}(\bR)$.

\begin{lem}
  Let $c\in(0,c_{0})$. Then $\mathcal{P}_{c} \neq \emptyset$ holds. \label{lem:nonempty_of_constraint}
\end{lem}
\begin{proof}
  First, we show that there exists $v\in H^{1}(\bR)$ such that $P_{c}(v)<0$.
  We put
  \begin{align}
    G_{c}(s) := \frac{c}{2}s^{2} - \frac{1}{p+1}|s|^{p+1} + \frac{1}{q+1}|s|^{q+1}, \ s\in\bR.
  \end{align}
  If $c\in(0,c_{0})$, we can see that there exists $s_{0}>0$ such that $G_{c}(s_{0})<0$.
  Under this condition, we can find $v_{0}\in H^{1}(\bR)$ satisfying $P_{c}(v_{0}) < 0$ (see the proof of Theorem 2 of \cite{Berestycki-Lions-1}.)

  For case $\sigma=1$, considering the graph of the function $(0,\infty) \ni \lambda \mapsto I_{c}(\lambda v_{0})$ yields that there exists $\lambda_{0} > 0$ such that $P_{c}(\lambda_{0} v_{0}) = 0$, which implies that $\lambda_{0} v_{0}$ is a desired function to find.

  If $0 < \sigma < 1$, there exist $v_{0}\in H^{1}(\bR)$ and $\alpha>0$ such that $P_{c}(v_{0}) = \alpha$, from the discussion in the first half.
  Here we set $v_{0}^{\alpha}(x) := v_{0}(\alpha x)$ for $x\in\bR$ so that $P_{c}(v_{0}^{\alpha}) = 1$.
  Therefore, $v_{0}^{\alpha}$ is a desired function.
\end{proof}

Now, we state the existence of a minimizer of \eqref{eq:minimizing_problem}.

\begin{prop}
  Let $0 < \sigma \leq 1$.
  If $c\in(0,c_{0})$, then there exists a minimizer of \eqref{eq:minimizing_problem} which is nonnegative, even, and decreasing in $|x|$. \label{prop:existense_of_minimizer_of_j}
\end{prop}

Here we introduce two lemmas related to the compactness used in the proof.

\begin{lem}
  Let $r\in[1,\infty)$ and $u\in L^{r}(\bR)$ be nonnegative, even, and decreasing in $|x|$.
  Then
  \begin{align}
    u(x) \leq 2^{-1/r} |x|^{-1/r} \norm{u}{L^{r}}
  \end{align}
  holds for almost all $x\in \bR$.  \label{lem:Radial}
\end{lem}
\begin{proof}
  Without loss of generality, we may assume that $x>0$.
  A direct calculation yields
  \begin{align}
    \norm{u}{L^{r}}^{r} = 2 \int_{0}^{\infty} u(y)^{r} \ dy \geq 2\int_{0}^{x} u(y)^{r} \ dy \geq 2u(x)^{r} x,
  \end{align}
  which implies the desired inequality.
\end{proof}

\begin{lem}[Strauss\cite{Strauss}]
  Assume that $P, \ Q\colon \bR \rightarrow \bR$ are continuous functions satisfying
  \begin{align}
    \frac{P(s)}{Q(s)} \rightarrow 0 \ \text{\rm as} \ |s| \rightarrow \infty \ \text{\rm and} \ |s| \rightarrow 0.
  \end{align}
  Moreover, let $(u_{n})_{n}$ be a sequence of measurable functions in $\bR$ which satisfies
  \begin{align}
    \sup_{n\in\bN}\int_{\bR}|Q(u_{n})| \, dx < \infty
  \end{align}
  and assume that there exists a measurable function $v\colon \bR \rightarrow \bR$ such that $P(u_{n}(x))\rightarrow v(x)$ a.e.\ in $\bR$, and $u_{n}(x)\rightarrow 0$ as $|x|\rightarrow 0$ uniformly respect to $n\in\bN$.
  Then we obtain $P(u_{n})\rightarrow v$ in $L^{1}(\bR)$. \label{lem:Strauss}
\end{lem}

Here we prove Proposition \ref{prop:existense_of_minimizer_of_j}.

\begin{proof}[Proof of Proposition \ref{prop:existense_of_minimizer_of_j}]
  We split the proof into two cases with respect to $\sigma$.

  \underline{Case 1. $\sigma=1$}:
  Let $(u_{n})_{n}\subset H^{1/2}(\bR)$ be a minimizing sequence of $j_{3}$ so that we have $S_{3}(u_{n}) \rightarrow j_{3}$ and $P_{c}(u_{n})=0$ for all $n\in\bN$.
  By a certain scaling, we may assume that $\norm{u_{n}}{L^{2}} = 1$ for all $n\in\bN$.
  Therefore, $(u_{n})_{n}$ is bounded in $H^{1/2}(\bR)$.
  Here we take the symmetric decreasing rearrangement of $(u_{n})_{n}$, denoting $(u^{\ast}_{n})_{n}$.
  Then $(u^{\ast}_{n})_{n}$ is also bounded in $H^{1/2}(\bR)$ and we have $J_{3}(u^{\ast}_{n}) \leq J_{3}(u_{n}), \ P_{c}(u^{\ast}_{n}) = 0, \ \norm{u^{\ast}_{n}}{L^{2}} = 1$ for all $n\in\bN$.
  Taking a subsequence of $(u^{\ast}_{n})_n$ and the Rellich compact embedding yield that there exists $v_{0}\in H^{1/2}(\bR)$ such that $u^{\ast}_{n} \rightharpoonup v_{0}$ weakly in $H^{1/2}(\bR)$ and $u^{\ast} \rightarrow v_{0}$ a.e.\ in $\bR$.
  The second convergence means that we can take the weak limit $v_{0}$ as even, nonnegative, and decreasing in $|x|$.

  In the following, we show that $v_{0}$ attains \eqref{eq:minimizing_problem}.
  First, we obtain
  \begin{align}
    J_{3}(v_{0}) \leq \liminf_{n\rightarrow\infty} J_{3}(u^{\ast}_{n}) \leq \liminf_{n\rightarrow\infty} J_{3}(u_{n}) = j_{3}, \label{eq:B_4}
  \end{align}
  by the lower semicontinuity of the norms.
  Next, by applying Lemma \ref{lem:Radial} to $(u^{\ast}_{n})_{n}$, there exists $C > 0$ such that $|u^{\ast}_{n}(x)| \leq C |x|^{-1/(q+1)}$ for all $n\in\bN$, which means that $u^{\ast}_{n} \rightarrow 0$ as $|x|\rightarrow\infty$ uniformly respect to $n\in\bN$.
  Here we let $r \in (1,\infty)$ and put $P(s):=|s|^{r+1}, \ Q(s):=s^{2} + |s|^{r+2}$ for $s\in\bR$ and $w_{n}:=u^{\ast}_{n}-v_{0}$ for $n\in\bN$.
  Then we can see that the functions $P$ and $Q$ and the sequence $(w_{n})_{n}$ satisfy the assumptions of the Lemma \ref{lem:Strauss}.
  Therefore, we obtain $\norm{w_{n}}{L^{r+1}}^{r+1} \rightarrow 0$.
  Furthermore, by $P_{c}(u^{\ast}_{n})=0$ and $\norm{u^{\ast}_{n}}{L^{2}}=1$, we have
  \begin{align}
    \frac{c}{2} = P_{c}(u^{\ast}_{n}) + \frac{c}{2}\norm{u^{\ast}_{n}}{L^{2}}^{2} = \frac{1}{p+1}\norm{u^{\ast}_{n}}{L^{p+1}}^{p+1} - \frac{1}{q+1}\norm{u^{\ast}_{n}}{L^{q+1}}^{q+1}. \label{eq:B_5}
  \end{align}
  Combining \eqref{eq:B_5} with the convergence yields
  \begin{align}
    \frac{c}{2} = \frac{1}{p+1}\norm{v_{0}}{L^{p+1}}^{p+1} - \frac{1}{q+1}\norm{v_{0}}{L^{q+1}}^{q+1} = P_{c}(v_{0}) + \frac{c}{2}\norm{v_{0}}{L^{2}}^{2},
  \end{align}
  which implies $v_{0}\neq 0$ and this concludes $j_{3}>0$.
  Moreover, by the Fatou lemma, we have $P_{c}(v_{0}) \geq 0$.
  Here we see that $P_{c}(v_{0})=0$ by contradiction.
  Assume $P_{c}(v_{0}) > 0$ and we consider the function
  \begin{align}
    (0,\infty) \ni \lambda \mapsto P_{c}(\lambda v_{0}) = -\frac{c}{2}\lambda^{2}\norm{v_{0}}{L^{2}}^{2} + \frac{\lambda^{p+1}}{p+1}\norm{v_{0}}{L^{p+1}}^{p+1} - \frac{\lambda^{q+1}}{q+1}\norm{v_{0}}{L^{q+1}}^{q+1}.
  \end{align}
  The shape of the graph gives $\lambda_{0}\in(0,1)$ such that $P_{c}(\lambda_{0} v_{0}) = 0$.
  However, $J_{3}(\lambda_{0} v_{0}) = \lambda_{0}^{2}J_{3}(v_{0}) \leq \lambda_{0} j_{3} < j_{3}$, which contradicts to the definition of $j_{3}$.
  Hence, we have $P_{c}(v_{0})=0$.
  Thus, we finally obtain $J_{3}(v_{0})=j_{3}$, which implies that $v_{0}$ is a desired minimizer.
  
  \underline{Case 2. $0 < \sigma < 1$}:
  Let $(u_{n})_{n}\subset H^{\sigma/2}(\bR)$ be a minimizing sequence of $j_{3}$, that is, $S_{3}(u_{n})\rightarrow j_{3}$, $P_{c}(u_{n})=0$ for all $n\in\mathbb{N}$.
  First, we show that $(u_{n})_{n}$ is bounded in $H^{\sigma/2}(\bR)$.
  We can see that there exists $C_{0}>0$ such that
  \begin{align}
      s^{p} - s^{q} \leq \frac{c}{2}s + C_{0}s^{2^{\ast}_{\sigma}-1}
  \end{align}
  for $s>0$.
  By integration, we have
  \begin{align}
      \frac{1}{p+1}s^{p+1} - \frac{1}{q+1}s^{q+1} \leq \frac{c}{4}s^{2} + Cs^{2^{\ast}_{\sigma}} \label{eq:B_9}
  \end{align}
  for $s>0$.
  From $P_{c}(u_{n})=0$, we have
  \begin{align}
    1 + \frac{c}{2}\norm{u_{n}}{L^{2}}^{2} = \frac{1}{p+1}\norm{u_{n}}{L^{p+1}}^{p+1} - \frac{1}{q+1}\norm{u_{n}}{L^{q+1}}^{q+1}. \label{eq:B_10}
  \end{align}
  Applying \eqref{eq:B_9} and the Sobolev inequality, we obtain
  \begin{align}
      \frac{1}{p+1}\norm{u_{n}}{L^{p+1}}^{p+1} - \frac{1}{q+1}\norm{u_{n}}{L^{q+1}}^{q+1} &\leq \frac{c}{4}\norm{u_{n}}{L^{2}}^{2} + C\norm{D_{x}^{\sigma/2}u_{n}}{L^{2}}^{2} \\
      &\leq \frac{c}{4}\norm{u_{n}}{L^{2}}^{2} + C. \label{eq:B_11}
  \end{align}
  Combining \eqref{eq:B_10} and \eqref{eq:B_11} yields
  \begin{align}
      \frac{c}{4}\norm{u_{n}}{L^{2}}^{2} \leq C,
  \end{align}
  which implies that $(u_{n})_{n}$ is bounded in $H^{\sigma/2}(\bR)$.
  Therefore, the symmetric decreasing rearrangement $(u^{\ast}_{n})_{n}$ of $(u_{n})_{n}$ is also bounded in $H^{\sigma/2}(\bR)$.
  By taking a subsequence, we obtain $v_{0}\in H^{\sigma/2}(\bR)$ such that $u^{\ast}_{n} \rightharpoonup v_{0}$ weakly in $H^{\sigma/2}(\bR)$ and $u^{\ast}_{n} \rightarrow v_{0}$ a.e.\ in $\bR$.
  Same as Case 1, $v_{0}$ can be taken as even, nonnegative, and decreasing in $|x|$.

  Next, we prove that $v_{0}$ is a desired minimizer.
  By the Fatou lemma and Lemmas \ref{lem:Radial}, \ref{lem:Strauss}, we obtain
  \begin{align}
    \frac{c}{2}\norm{v_{0}}{L^{2}}^{2} &\leq \liminf_{n\rightarrow\infty} \left( \frac{c}{2}\norm{u^{\ast}_{n}}{L^{2}}^{2} \right) \\
      &= \liminf_{n\rightarrow\infty} \left( \frac{1}{p+1}\norm{u^{\ast}_{n}}{L^{p+1}}^{p+1} - \frac{1}{q+1}\norm{u^{\ast}_{n}}{L^{q+1}}^{q+1} -1 \right) \\
      &= \frac{1}{p+1}\norm{v_{0}}{L^{p+1}}^{p+1} - \frac{1}{q+1}\norm{v_{0}}{L^{q+1}}^{q+1} -1,
  \end{align}
  which implies $P_{c}(v_{0}) \geq 1$.
  Furthermore, we have $J_{3}(v_{0}) \leq j_{3}$ by the lower semicontinuity of the norms.
  Therefore, we can see $P_{c}(v_{0}) = 1$.
  Indeed, assume that $P_{c}(v_{0}) > 1$.
  Then there exists $\alpha\in(0,1)$ such that $P_{c}(v_{0}^{\alpha}) = 1$, where $v_{0}^{\alpha}(x) := v_{0}(x/\alpha)$.
  Therefore, we have $J_{3}(v_{0}^{\alpha}) = \alpha^{1-\sigma} J_{3}(v_{0}) < j_{3}$, which contradicts to the definition of $j_{3}$.
  Hence, $v_{0}$ is a desired minimizer.
\end{proof}

Now, we show Proposition \ref{prop:existence_of_ground_state_of_SP3}.

\begin{proof}[Proof of Propositon \ref{prop:existence_of_ground_state_of_SP3}]

  Let $v_{0}\in H^{\sigma/2}(\bR)$ be a minimizer of \eqref{eq:minimizing_problem} which is nonnegative, even, and decreasing in $|x|$.
  Then there exists a Lagrange multiplier $\mu\in\bR$ such that $J_{3}'(v_{0}) = \mu P_{c}'(v_{0})$, that is
  \begin{align}
    D_{x}^{\sigma}(v_{0}) = \mu (-cv_{0} + |v_{0}|^{p-1}v_{0} - |v_{0}|^{q-1}v_{0}), \label{eq:B_14}
  \end{align}
  Here we show $\mu > 0$.
  Multiplying $v_{0}$ to both hand sides of \eqref{eq:B_14} yields
  \begin{align}
    0 < \norm{D_{x}^{\sigma/2}v_{0}}{L^{2}}^{2} = \mu\dualC{P_{c}'(v_{0})}{v_{0}}, \label{eq:B_16}
  \end{align}
  which implies that $\mu \neq 0$ and
  \begin{align}
    \dualC{P_{c}'(v_{0})}{v_{0}} \neq 0. \label{eq:B_15}
  \end{align}
  Now, we consider the graph of the function $(0,\infty) \ni \lambda \mapsto P_{c}(\lambda v_{0})$.
  Thanks to \eqref{eq:B_15}, we can find two points $0 < \lambda_{1} < \lambda_{2} < \infty$ which satisfy $P_{c}(\lambda_{j} v_{0})= 0$ for $j=1, \ 2$ and either of whom equals to $1$.
  Then we can see that $\lambda_{1}=1$.
  Indeed, if we assume that $\lambda_{2}=1$, the same contradiction occurs as in the proof of Case 1 of Proposition \ref{prop:existense_of_minimizer_of_j}.
  Therefore, we obtain $\dualC{P_{c}'(v_{0})}{v_{0}} > 0$, which implies $\mu > 0$ by \eqref{eq:B_16}.
  Then we set $\phi(x) := v_{0}(\mu^{1/\sigma}x)$ for $x\in\bR$ so that $\phi$ is a solution to \eqref{eq:SP3}.

  In case $\sigma = 1$, we can see that $\norm{D_{x}^{1/2}\phi}{L^{2}}^{2} = \norm{D_{x}^{1/2}v_{0}}{L^{2}}^{2}$ and $P_{c}(\phi)=P_{c}(v_{0})=0$, which means $S_{3}(\phi) = S_{3}(v_{0}) = J_{3}(v_{0}) = j_{3}$ and then $\phi$ is a ground state of \eqref{eq:SP3}.

  If $0 < \sigma <1$, we use the Pohozaev identity to find a ground state.
  First, since $\phi\in\mathcal{N}_{3}$, we obtain
  \begin{align}
    \frac{1-\sigma}{2} \norm{D_{x}^{\sigma/2}\phi}{L^{2}}^{2} = P_{c}(\phi). \label{eq:B_51}
  \end{align}
  by Lemma \ref{lem:Pohozaev}.
  Moreover, we have
  \begin{align}
    & \frac{1-\sigma}{2}\norm{D_{x}^{\sigma/2}\phi}{L^{2}}^{2} = \frac{1-\sigma}{2} \mu^{(\sigma - 1) / \sigma} \norm{D_{x}^{\sigma/2}v_{0}}{L^{2}}^{2} = \frac{1-\sigma}{\sigma}\mu^{(\sigma-1) / \sigma}j_{3}, \label{eq:B_52} \\
    & P_{c}(\phi) = \mu^{-1/\sigma} P_{c}(v_{0}) = \mu^{-1/\sigma}. \label{eq:B_53}
  \end{align}
  Combining \eqref{eq:B_51}, \eqref{eq:B_52}, and \eqref{eq:B_53} yields
  \begin{align}
    \mu = \left( \frac{1-\sigma}{\sigma} \right)^{-1} j_{3}^{-1}. \label{eq:B_54}
  \end{align}
  Using \eqref{eq:B_54}, we obtain
  \begin{align}
    S_{3}(\phi) = J_{3}(\phi) = \mu^{(\sigma-1) / \sigma} j_{3} = \left( \frac{1-\sigma}{\sigma} \right)^{(1-\sigma) / \sigma} j_{3} ^{1/\sigma}. \label{eq:B_55}
  \end{align}
  Next, let $u\in\mathcal{N}_{3}$.
  By the Pohozaev identity, we can see that 
  \begin{align}
    P_{c}(u) = \frac{1-\sigma}{2}\norm{D_{x}^{\sigma/2}u}{L^{2}}^{2} = \frac{1-\sigma}{\sigma} J_{3}(u)
  \end{align}
  and $P_{c}(u^{\alpha}) = 1$, where $\alpha := P_{c}(u)^{-1} > 0$ and $u^{\alpha}(x) := u(x/\alpha)$ for $x\in\bR$, which means $u^{\alpha}\in\mathcal{P}_{c}$.
  Moreover, we have
  \begin{align}
    J_{3}(u^{\alpha}) = \alpha^{1-\sigma} J_{3}(u) = P_{c}(u)^{\sigma-1} J_{3}(u) = \left( \frac{1-\sigma}{\sigma} \right)^{\sigma -1} J_{3}(u)^{\sigma},
  \end{align}
  which implies
  \begin{align}
    S_{3}(u) = J_{3}(u) = \left( \frac{1-\sigma}{\sigma} \right)^{(1-\sigma) / \sigma} J_{3}(u^{\alpha})^{1/\sigma}. \label{eq:B_56}
  \end{align}
  Finally, considering \eqref{eq:B_55}, \eqref{eq:B_56}, and the definition of $j_{3}$ yield $S_{3}(\phi) \leq S_{3}(u)$ for all $u\in\mathcal{N}_{3}$.
  This means that $\phi$ is a ground state of \eqref{eq:SP3}.
\end{proof}

\section{Properties of solutions} \label{section:properties}
In this section, we consider the regularity and positivity of solutions to \eqref{eq:SPj}.

First, we introduce a function which plays an important role in this section.
For $0<\sigma<2$ and $\nu>0$, we define the function $N_{\nu}^{\sigma}\colon \bR \rightarrow \bR$ as
\begin{align}
  N_{\nu}^{\sigma}(x) := \frac{1}{2\pi} \int_{\bR} \frac{1}{|\xi|^{\sigma} + \nu} e^{i\xi x} \, d\xi.
\end{align}

Properties of $N_{\nu}^{\sigma}$ have been investigated in \cite[Appendix A]{Frank-Lenzmann}.
Here we recall them.

\begin{lem}[Frank--Lenzmann\cite{Frank-Lenzmann}]
  The following properties hold:
  \begin{enumerate}
    \item $N_{\nu}^{\sigma} \in L^{r}(\bR)$ for $r \in \left( 1, 1/(1-\sigma)_{+} \right)$.
    \item $N_{\nu}^{\sigma}$ is positive, even, and strictly decreasing in $|x|$.
  \end{enumerate} \label{lem:properties_of_integral_kernel}
\end{lem}

\subsection{Regularity of solutions}
Here we consider the regularity results of solutions to \eqref{eq:SPj}.

\begin{prop}
  Let $\phi\in H^{\sigma/2}(\bR)$ be a solution to \eqref{eq:SPj}.
  Then we obtain $\phi\in H^{\sigma +1}(\bR)$. \label{prop:regularity_of_solution}
\end{prop}

The same statement for \eqref{eq:SP0} with single power nonlinearities is proven by Frank--Lenzmann\cite[Appendix B]{Frank-Lenzmann}.
We can show Proposition \ref{prop:regularity_of_solution} with some modifications.

\begin{proof}
  Before proceeding, we remark that taking the Fourier transform to both sides of \eqref{eq:SPj}, we have
  \begin{align}
    \hat{\phi}(\xi) = \frac{C_{0}}{|\xi|^{\sigma} + c} \FT [f_{j}(\phi)](\xi) \label{eq:AppB_02}
  \end{align}
  with some constant $C_{0}>0$.

  First, we show $\phi\in H^{1}(\bR)$.
  The proof of this statement depends on the value of $\sigma$.

  \underline{Case 1. $1 < \sigma <2$:}
  First, we easily see that $\phi\in L^{\infty}(\bR)$ by the Sobolev embedding.
  Then using \eqref{eq:AppB_02}, we have
  \begin{align}
    \norm{D_{x}^{\sigma}\phi}{L^{2}} = \norm{|\xi|^{\sigma}\hat{\phi}}{L^{2}} &= C_{0} \left\| \frac{|\xi|^{\sigma}}{|\xi|^{\sigma}+c} \FT [f_{j}(\phi)] \right\|_{L^{2}} \\
    &\leq C_{0} \left( \norm{\FT [ |\phi|^{p-1}\phi ]}{L^{2}} + \norm{\FT [|\phi|^{q-1}\phi]}{L^{2}} \right) \\
    &\leq C_{0} \left( \norm{\phi}{L^{\infty}}^{(p-1)/2} + \norm{\phi}{L^{\infty}}^{(q-1)/2} \right) \norm{\phi}{H^{\sigma/2}},
  \end{align}
  which implies $\phi\in H^{\sigma}(\bR) \hookrightarrow H^{1}(\bR)$.

  \underline{Case 2. $\sigma =1$:}
  By \eqref{eq:SPj} and the Sobolev embedding $H^{1/2}(\bR) \hookrightarrow L^{r}(\bR)$ for $r\in[2,\infty)$, we have
  \begin{align}
    \norm{|\xi|\phi}{L^{2}} = \norm{D_{x}\phi}{L^{2}} &= \norm{-c \phi + f_{j}(\phi)}{L^{2}} \leq c \norm{\phi}{L^{2}} + \norm{\phi}{L^{2p}}^{p} + \norm{\phi}{L^{2q}}^{q} \\
    &\leq C (1 + \norm{\phi}{H^{1/2}}^{p-1} + \norm{\phi}{H^{1/2}}^{q-1}) \norm{\phi}{H^{1/2}},
  \end{align}
  where $C>0$ is a constant dependent on $c, \ p$, and $q$.
  This concludes $\phi\in H^{1}(\bR)$.

  \underline{Case 3. $0 < \sigma < 1$:}
  First, we claim that there exists some $r_{1}, r_{2}\in (1,1/(1-\sigma))$ which satisfy
  \begin{align}
    \frac{1}{r_{1}} + \frac{p}{2^{\ast}_{\sigma}} = \frac{1}{r_{1}} + \frac{p(1-\sigma)}{2} = 1, \label{eq:AppB_03} \\
    \frac{1}{r_{2}} + \frac{q}{2^{\ast}_{\sigma}} = \frac{1}{r_{2}} + \frac{q(1-\sigma)}{2} = 1. \label{eq:AppB_04}
  \end{align}
  Indeed, since $1 < p < q < 2^{\ast}_{\sigma}-1$, we can see that
  \begin{align}
    1 < \frac{2}{p(1-\sigma)}, \ \frac{2}{q(1-\sigma)} < \frac{2}{1-\sigma} = 2^{\ast}_{\sigma}.
  \end{align}
  Now, we put
  \begin{align}
    r_{1} := \left( 1 - \frac{p(1-\sigma)}{2} \right)^{-1} = \frac{2}{2-p(1-\sigma)}.
  \end{align}
  Then by $0 < \sigma < 1$ and $p < 2^{\ast}_{\sigma} -1$, we can see that $r_{1} \in (1,1/(1-\sigma))$.
  We can take some $r_{2}$ in the same way as taking $r_{1}$.

  By the way, we can write a solution $\phi$ of \eqref{eq:SPj} as $\phi = C_{0} N^{\sigma}_{c} \ast f_{j}(\phi)$ with some constant $C_{0}>0$.
  Therefore, combining this and \eqref{eq:AppB_03}, \eqref{eq:AppB_04}, we obtain
  \begin{align}
    \norm{\phi}{L^{\infty}} = C_{0} \norm{N_{c}^{\sigma} \ast f_{j}(\phi)}{L^{\infty}} &\leq C_{0} \left( \norm{N_{c}^{\sigma} \ast |\phi|^{p-1}|\phi|}{L^{\infty}} + \norm{N_{c}^{\sigma} \ast |\phi|^{q-1}|\phi|}{L^{\infty}}  \right) \\
    &\leq C_{0} \left( \norm{N_{c}^{\sigma}}{L^{r_{1}}} + \norm{N_{c}^{\sigma}}{L^{r_{2}}} \right) \norm{\phi}{L^{2^{\ast}_{\sigma}}} \\
    &\leq C \left( \norm{N_{c}^{\sigma}}{L^{r_{1}}} + \norm{N_{c}^{\sigma}}{L^{r_{2}}} \right) \norm{\phi}{H^{\sigma/2}},
  \end{align}
  where $C>0$ depends on $\sigma$.
  Hence, $\phi\in L^{\infty}(\bR)$ follows from this inequality and Lemma \ref{lem:properties_of_integral_kernel}.

  Next, we show $\phi\in H^{1}(\bR)$.
  Using \eqref{eq:AppB_02}, we obtain
  \begin{align}
    \norm{D_{x}^{(k+1)\sigma/2}\phi}{L^{2}} &= C_{0} \left\| \frac{|\xi|^{(k+1)\sigma/2}}{|\xi|^{\sigma}+c} \FT[f_{j}(\phi)] \right\|_{L^{2}} \\
    &\leq C \left( \left\| |\xi|^{k\sigma/2} \FT[|\phi|^{p-1}\phi] \right\|_{L^{2}} + \left\| |\xi|^{k\sigma/2} \FT[|\phi|^{q-1}\phi] \right\|_{L^{2}} \right) \\
    &\leq C \left( \norm{\phi}{L^{\infty}}^{(p-1)/2} + \norm{\phi}{L^{\infty}}^{(q-1)/2} \right) \norm{|\xi|^{k\sigma/2}\hat{\phi}}{L^{2}} \\
    &\leq C \norm{D_{x}^{k\sigma/2}\phi}{L^{2}} \label{eq:AppB_05}
  \end{align}
  for all $k\in\mathbb{N}$, where $C>0$ depends on $\sigma, \ p$, and $q$.
  Here, we take some $n\in\mathbb{N}$ such that $1/(n+1)\leq \sigma/2 < 1/n$.
  Then using \eqref{eq:AppB_05} infinitely many times, we have $\norm{D_{x}^{(n+1)\sigma/2}}{L^{2}} \leq C \norm{D_{x}^{\sigma/2}\phi}{L^{2}} \leq C \norm{\phi}{H^{\sigma/2}}$, which implies $\phi\in H^{(n+1)\sigma/2}(\bR) \hookrightarrow H^{1}(\bR)$.

  Finally, we show that $\phi\in H^{\sigma + 1}(\bR)$.
  Using \eqref{eq:AppB_02}, we have
  \begin{align}
    \norm{D_{x}^{\sigma+1}\phi}{L^{2}} = \norm{|\xi|^{\sigma+1}\phi}{L^{2}} &\leq \norm{|\xi|\FT[|\phi|^{p-1}\phi]}{L^{2}} + \norm{|\xi|\FT[|\phi|^{q-1}\phi]}{L^{2}} \\
    &= \norm{\partial_{x}(|\phi|^{p-1}\phi)}{L^{2}} + \norm{\partial_{x}(|\phi|^{p-1}\phi)}{L^{2}} \\
    &\leq \left( p\norm{\phi}{L^{\infty}}^{(p-1)/2} + q\norm{\phi}{L^{\infty}}^{(q-1)/2} \right) \norm{\phi}{H^{1}},
  \end{align}
  which means $\phi \in H^{\sigma+1}(\bR)$.
  This completes the proof.
\end{proof}

\subsection{Positivity of solutions}

In this subsection, we consider the positivity of solutions to \eqref{eq:SPj}.
\begin{prop}
  Let $\phi\in H^{\sigma/2}(\bR)$ be a nonnegative solution to \eqref{eq:SPj}.
  Then $\phi$ is strictly positive. \label{prop:positiveness_of_GS}
\end{prop}
\begin{proof}
  For simplicity, we only consider \eqref{eq:SP1}.
  We can prove this for the remaining cases in almost the same way.
  Let $\phi\in H^{\sigma/2}(\bR)$ be a nonnegative solution of \eqref{eq:SP1}, and we can see that $\phi\in H^{\sigma + 1}(\bR)$ from Proposition \ref{prop:regularity_of_solution}, which implies $\phi\in L^{\infty}(\bR)$ and $\phi$ is continuous in $\bR$.
  Therefore, we put
  \begin{align}
    \lambda_{1} := -\min_{x\in\bR} (|\phi(x)|^{p-1}-|\phi(x)|^{q-1}) + 1,
  \end{align}
  so that we obtain $ \lambda_{1} + |\phi(y)|^{p-1} - |\phi(y)|^{q-1} > 0 $ for all $y\in\bR$.
  Then by adding $\lambda_{1}\phi$ to both sides of \eqref{eq:SP1} and transforming, we obtain
  \begin{align}
    \phi(x) = C \int_{\bR} N_{c+\lambda_{1}}^{\sigma}(x-y)(\lambda_{1} + |\phi(y)|^{p-1} - |\phi(y)|^{q-1})\phi(y) \ dy. \label{eq:C_02}
  \end{align}
  Then we can see that $\phi>0$ in $\bR$.
  Indeed, if there exists $x_{0}\in\bR$ such that $\phi(x_{0})=0$, then,  by \eqref{eq:C_02}, we have
  \begin{align}
    \int_{\bR} N_{c+\lambda_{1}}^{\sigma}(x_{0}-y)(\lambda_{1} + |\phi(y)|^{p-1} - |\phi(y)|^{q-1})\phi(y) \ dy = 0. \label{eq:C_03}
  \end{align}
  Since $N_{c+\lambda_{1}}^{\sigma}(x_{0}-y)(\lambda_{1} + |\phi(y)|^{p-1} - |\phi(y)|^{q-1}) > 0$ for all $y\in\bR$, we obtain $\phi\equiv 0$ in $\bR$ from \eqref{eq:C_03}, which contradicts to that $\phi$ is nontrivial.
  This concludes the proof.
\end{proof}

By Proposition \ref{prop:positiveness_of_GS}, we can obtain more properties of ground states.
\begin{prop}
  Set $j=1, \ 2$.
  Then there do not exist any ground states of \eqref{eq:SPj} with sign changes. \label{prop:no_sign_changes}
\end{prop}
\begin{proof}
  For simplicity, we only consider the case $j=1$.
  We show this corollary with contradiction.

  Assume that there exists $\phi\in\mathcal{G}_{1}$ with sign changes.
  Then we can see that $|\phi|\in\mathcal{G}_{1}$ similarly to Proposition \ref{prop:even_ground_state_of_SP1_2} by using \eqref{eq:A_101}.
  Therefore, by Proposition \ref{prop:positiveness_of_GS}, $|\phi|$ is strictly positive.
  However, by the assumption of $\phi$, there exists $x_{0}\in\bR$ such that $|\phi(x_{0})|=0$.
  This is contradiction.
\end{proof}

Here we conclude the proof of Theorems \ref{Thm:SP1_2} and \ref{Thm:SP3}.
Theorem \ref{Thm:SP1_2} follows from Propositions \ref{prop:existense_of_gs_12}, \ref{prop:regularity_of_solution}, \ref{prop:positiveness_of_GS}, and \ref{prop:no_sign_changes}.
Theorem \ref{Thm:SP3} follows from Propositions \ref{prop:existence_of_ground_state_of_SP3}, \ref{prop:regularity_of_solution}, and \ref{prop:positiveness_of_GS}.

\section{Application and specific phenomenon} \label{section:application}

In this section, we consider Theorem \ref{thm:classification}.
Here we recall the equation \eqref{eq:SP}:
\begin{align}
  D_{x}\phi + c\phi + \phi^{p} - \phi^{q} = 0, \quad x\in\bR, \tag{SP} \label{eq:SP}
\end{align}
where $c>0, \ p, \ q \in \mathbb{N}, \ 2 \leq p < q$.
The action functional corresponding to \eqref{eq:SP} is
\begin{align}
  S(u) = \frac{1}{2}\norm{D_{x}^{1/2}u}{L^{2}}^{2} + \frac{c}{2}\norm{u}{L^{2}}^{2} + \frac{1}{p+1}\int_{\bR} u^{p+1} \, dx - \frac{1}{q+1} \int_{\bR} u^{q+1} \, dx.
\end{align}

The existence of ground states is obtained by the method in \S \ref{subsection:Nehari} with certain corrections.
The remaining statements of the non/existence come from the main results of this paper.
For example, in Case (\ref{case:even_even}), a positive ground state of \eqref{eq:SP1} becomes also a positive solution to \eqref{eq:SP}.
Furthermore, if the assumptions of Theorem \ref{Thm:SP3} are satisfied, we can find a negative solution to \eqref{eq:SP} as below;
let $v\in H^{1/2}(\bR)$ be a positive solution to \eqref{eq:SP3} and put $u := -v$.
Then $u$ becomes a negative solution to \eqref{eq:SP}.

In the following, we prove the statement in Case (\ref{case:even_odd}) that none of the positive solutions is a ground state.

\begin{proof}

  First, we recall that $\phi_{1}\in\mathcal{G}_{1}$ attains $d_{1} = \inf_{v\in\mathcal{K}_{1}}S_{1}(v)$.
  Moreover, we define notations corresponding to \eqref{eq:SP} which are similar to those in \S\ref{subsection:Nehari}:
  \begin{align}
    &\mathcal{N} := \set{v\in H^{1/2}(\bR) \setminus \{ 0 \}}{S'(v)=0}, \\
    &K(u) := \dualC{S'(u)}{u} = \norm{D_{x}^{1/2}u}{L^{2}}^{2} + c\norm{u}{L^{2}}^{2} + \int_{\bR} u^{p+1} \, dx - \int_{\bR} u^{q+1} \, dx, \\
    &\mathcal{K} := \set{ v\in H^{1/2}(\bR)\setminus \{ 0 \} }{ K(v) = 0 }, \quad d := \inf_{v\in\mathcal{K}}S(v), \\
    &\mathcal{M} := \set{ v\in H^{1/2}(\bR) }{ S(v) = d }.
  \end{align}
  Then as in Lemma \ref{lem:relation_G_and_M}, we can see that $\mathcal{M}$ coincides with the set of ground states of \eqref{eq:SP}, which we denote $\mathcal{G}$ in the following.
  Next, we put
  \begin{align}
    \mathcal{A} := \set{v\in\mathcal{N}}{v>0 \ \text{in} \ \bR}, \quad a := \inf_{v\in\mathcal{A}}S(v).
  \end{align}
  To prove the statement, it suffices to see that $d < a$.

  First, we show that $d < d_{1}$.
  Here, we let $\phi\in\mathcal{G}$ and take some $\phi_{1}\in\mathcal{G}_{1}$ which is positive.
  Moreover, we put $\psi_{1}:=-\phi_{1}$ so that $\psi_{1}$ becomes a negative solution to \eqref{eq:SP}.
  Therefore, by the definitions of $d$ and $d_{1}$, we obtain
  \begin{align}
    d = S(\phi) \leq S(\psi_{1}) < S_{1}(\psi_{1}) = S_{1}(\phi_{1}) = d_{1}.
  \end{align}

  Next, we prove that $d_{1} \leq a$ holds.
  Here we take some $\phi_{1}\in\mathcal{G}_{1}$.
  In addition, for each $v\in\mathcal{A}$, we put $w:=-v$ so that $w$ becomes a solution to \eqref{eq:SP1}.
  Therefore,
  \begin{align}
    d_{1} = S_{1}(\phi_{1}) \leq S_{1}(w) = S_{1}(v) = S(v).
  \end{align}
  holds for all $v\in\mathcal{A}$, which implies $d_{1} \leq a$.
  This completes the proof.
\end{proof}

\appendix
\section{Proof of Lemma \ref{lem:Lieb}} \label{appendix:proof_of_Lieb}

Before starting the proof, we introduce the Sobolev space $\tilde{H}^{\sigma/2}(\Omega)$, where $\Omega$ is an open subset of $\bR$.
$\tilde{H}^{\sigma/2}(\Omega)$ is defined as
\begin{align}
  \tilde{H}^{\sigma/2}(\Omega) := \SET{ v\in L^{2}(\Omega) }{ \frac{|v(x)-v(y)|}{|x-y|^{\frac{1}{2} + \frac{\sigma}{2}}}\in L^{2}(\Omega) }
\end{align}
with a norm
\begin{align}
  \norm{u}{\tilde{H}^{\sigma/2}(\Omega)}^{2} := \norm{u}{L^{2}(\Omega)}^{2} + \iint_{\Omega\times\Omega} \frac{|u(x)-u(y)|^{2}}{|x-y|^{1+\sigma}} \, dxdy.
\end{align}
Then we can see that $\tilde{H}^{\sigma/2}(\bR) = H^{\sigma/2}(\bR)$ and $\norm{\cdot}{\tilde{H}^{\sigma/2}(\bR)} \sim \norm{\cdot}{H^{\sigma/2}}$.
For details, see \cite{HitchhikersOfFractionalSobolev}.

\begin{proof}[Proof of Lemma \ref{lem:Lieb}]
  First, we put $I_{z}:=(z,z+1)$ for $z\in\mathbb{Z}$.
  Then we see that
  \begin{align}
    \norm{u}{H^{\sigma/2}}^{2} \geq \sum_{z\in\mathbb{Z}}\norm{u}{\tilde{H}^{\sigma/2}(I_{z})}^{2} \label{eq:A_16}
  \end{align}
  holds for any $u\in H^{\sigma/2}(\bR)$.
  In addition, we set
  \begin{align}
    M:=\sup_{n\in\bN}\norm{u_{n}}{H^{\sigma/2}}, \quad m:= \inf_{n\in\bN}\norm{u_{n}}{L^{r}(\bR)}, \quad C_{0}:=\frac{M^{2}+1}{m^{r}}.
  \end{align}
  Then we claim the following statement \eqref{eq:A_17}:
  \begin{align}
    \text{For all} \ n\in\bN, \ \text{there exists} \ z_{n}\in\mathbb{Z} \ \text{which satisfies} \ \norm{u_{n}}{\tilde{H}^{\sigma/2}(I_{z})}^{2} < C_{0}\norm{u_{n}}{L^{r}(I_{z})}^{r}. \label{eq:A_17}
  \end{align}
  Indeed, if \eqref{eq:A_17} does not hold, there exists $n_{0}\in\bN$ which satisfies $\norm{u_{n_{0}}}{\tilde{H}^{\sigma/2}(I_{z})}^{2} \geq C_{0}\norm{u_{n_{0}}}{L^{r}(I_{z})}^{r}$ for any $z\in\mathbb{Z}$.
  Then using \eqref{eq:A_16} yields
  \begin{align}
    M^{2} \geq \norm{u_{n_{0}}}{H^{\sigma/2}}^{2} &\geq \sum_{z\in\mathbb{Z}}\norm{u_{n_{0}}}{\tilde{H}^{\sigma/2}(I_{z})}^{2} \\
    &\geq C_{0} \sum_{z\in\mathbb{Z}}\norm{u_{n_{0}}}{L^{r}(I_{z})}^{r} = C_{0}\norm{u_{n_{0}}}{L^{r}(\bR)}^{r} \geq C_{0}m^{r} = M^{2} + 1,
  \end{align}
  which is impossible to occur.
  Then from \eqref{eq:A_17}, we can choose a sequence $(z_{n})_{n}\subset\bR$.
  Here we set $v_{n}:=u_{n}(\cdot + z_{n})$ and then $(v_{n})_{n}$ is also bounded in $H^{\sigma/2}(\bR)$.
  Therefore, by taking a subsequence, we obtain $v\in H^{\sigma/2}(\bR)$ which satisfies $v_{n} \rightharpoonup v$ weakly in $H^{\sigma/2}(\bR)$.
  To complete the proof, we shall see that $v\neq 0$.
  By the Sobolev embedding and \eqref{eq:A_17}, we have
  \begin{align}
    \norm{v_{n}}{L^{r}(I_{0})}^{2} \leq C \norm{v_{n}}{\tilde{H}^{\sigma/2}(I_{0})}^{2} < C \norm{v_{n}}{L^{r}(I_{0})}^{r},
  \end{align}
  which means $\norm{v_{n}}{L^{r}(I_{0})}^{r-2} > C$.
  Furthermore, we obtain $v_{n} \rightarrow v$ in $L^{r}(I_{0})$ by the Rellich compact embedding.
  Thus, we have $\norm{v}{L^{r}(I_{0})}>0$, which means $v\neq 0$.
\end{proof}

\section*{Acknowledgements}
The author would like to thank Professor Masahito Ohta, Dr.\ Noriyoshi Fukaya, and Dr.\ Yoshinori Nishii for their support and helpful discussions for this study.

\bibliography{reference_list}

\providecommand{\MR}[1]{}
\providecommand{\bysame}{\leavevmode\hbox to3em{\hrulefill}\thinspace}
\providecommand{\MR}{\relax\ifhmode\unskip\space\fi MR }
\providecommand{\MRhref}[2]{%
  \href{http://www.ams.org/mathscinet-getitem?mr=#1}{#2}
}
\providecommand{\href}[2]{#2}
\begin{thebibliography}{10}

\bibitem{Angulo}
J.~Angulo~Pava, \emph{Nonlinear dispersive equations: existence and stability of solitary and periodic travelling wave solutions}, Mathematical Survays and Monographs, no. 156, Amer. Math. Soc., 2009.

\bibitem{Benjamin}
T.~B. Benjamin, \emph{Internal waves of permanent form in fluids of great depth}, J. Fluid Mech. \textbf{29} (1967), no.~3, 559--592.

\bibitem{BGK}
H.~Berestycki, T.~Gallou\"{e}t, and O.~Kavian, \emph{\'{E}quations de champs scalaires euclidiens non lin\'{e}aires dans le plan}, C. R. Acad. Sci. Paris S\'{e}r. I Math. \textbf{297} (1983), no.~5, 307--310. \MR{734575}

\bibitem{Berestycki-Lions-1}
H.~Berestycki and P.-L. Lions, \emph{Nonlinear scalar field equations. {I}. {E}xistence of a ground state}, Arch. Rational Mech. Anal. \textbf{82} (1983), no.~4, 313--345. \MR{695535}

\bibitem{Brezis-Lieb}
H.~Brezis and E.~Lieb, \emph{A relation between pointwise convergence of functions and convergence of functionals}, Proc. Amer. Math. Soc. \textbf{88} (1983), no.~3, 486--490. \MR{699419}

\bibitem{Chang-Wang}
X.~Chang and Z.-Q. Wang, \emph{Ground state of scalar field equations involving a fractional {L}aplacian with general nonlinearity}, Nonlinearity \textbf{26} (2013), no.~2, 479--494. \MR{3007900}

\bibitem{HitchhikersOfFractionalSobolev}
E.~Di~Nezza, G.~Palatucci, and E.~Valdinoci, \emph{Hitchhiker's guide to the fractional {S}obolev spaces}, Bull. Sci. Math. \textbf{136} (2012), no.~5, 521--573. \MR{2944369}

\bibitem{Frank-Lenzmann}
R.~L. Frank and E.~Lenzmann, \emph{Uniqueness of non-linear ground states for fractional {L}aplacians in {$\mathbb{R}$}}, Acta Math. \textbf{210} (2013), no.~2, 261--318. \MR{3070568}

\bibitem{Fukaya-Hayashi-2011}
N.~Fukaya and M.~Hayashi, \emph{Instability of algebraic standing waves for nonlinear {S}chr\"{o}dinger equations with double power nonlinearities}, Trans. Amer. Math. Soc. \textbf{374} (2021), no.~2, 1421--1447. \MR{4196398}

\bibitem{Lewin-Nodari}
M.~Lewin and S.-R. Nodari, \emph{The double-power nonlinear {S}chr\"{o}dinger equation and its generalizations: uniqueness, non-degeneracy and applications}, Calc. Var. Partial Differential Equations \textbf{59} (2020), no.~6, Paper No. 197, 49. \MR{4171857}

\bibitem{Lieb}
E.~H. Lieb, \emph{On the lowest eigenvalue of the {L}aplacian for the intersection of two domains}, Invent. Math. \textbf{74} (1983), no.~3, 441--448.

\bibitem{Ohta_1995_S}
M.~Ohta, \emph{Stability and instability of standing waves for one-dimensional nonlinear {S}chr\"{o}dinger equations with double power nonlinearity}, Kodai Math. J. \textbf{18} (1995), no.~1, 68--74. \MR{1317007}

\bibitem{Strauss}
W.~A. Strauss, \emph{Existence of solitary waves in higher dimensions}, Comm. Math. Phys. \textbf{55} (1977), no.~2, 149--162. \MR{454365}

\bibitem{Weinstein}
M.~I. Weinstein, \emph{Existence and dynamic stability of solitary wave solutions of equations arising in long wave propagation}, Comm. Partial Differential Equations \textbf{12} (1987), no.~10, 1133--1173. \MR{886343}

\end{thebibliography}
\bibliographystyle{amsplain}

\end{document}